\documentclass[a4paper,10pt]{article}
\usepackage{amsfonts}
\usepackage{amssymb}
\usepackage{amsthm}
\usepackage{cite}

\usepackage{graphicx}
\usepackage{amsfonts}
\usepackage{amssymb}
\usepackage{amsthm}
\usepackage{newlfont}
\usepackage{fancyhdr}
\usepackage{latexsym,amssymb,euscript}
\usepackage{rotating}
\usepackage{newlfont}
\input amssym.def
\input amssym
\def\dj{d\kern-0.4em\char"16\kern-0.1em}

 \newtheorem{thm}{Theorem}
 \newtheorem{cor}[thm]{Corollary}
 
 \newtheorem{prop}[thm]{Proposition}
\theoremstyle{definition}
 \newtheorem{defn}[thm]{Definition}
\newtheorem{exm}[thm]{Example}
 \newtheorem{rem}[thm]{Remark}
\title{ Shellability of complexes of directed trees }
\author{{ Du\v sko Joji\' c} \\
Faculty of Science, University of Banja Luka\\78 000 Banja Luka,
Bosnia and Herzegovina\\e-mail: ducci68@blic.net}

\date{}
\begin{document}
\maketitle
\begin{abstract}
 The question of shellability of
complexes of directed trees was asked by R. Stanley. D. Kozlov
showed that the existence of a complete source in a directed graph
provides a shelling of its complex of directed trees. We will show
that this property gives a shelling that is straightforward in
some sense. Among the simplicial polytopes, only the
crosspolytopes allow such a shelling. Furthermore, we show that
the complex of directed trees of a complete double directed graph
is a union of suitable spheres. We also investigate shellability
of the maximal pure skeleton of a complex of directed trees. Also,
we prove that the complexes of directed trees of a directed graph
which is essentially a tree is vertex-decomposable. For these
complexes we describe the set of generating facets.

\end{abstract}

\section{Introduction}

A \textit{directed tree} with a \emph{root} $r$ is an acyclic
directed graph $T=(V(T),E(T))$ such that for every $x\in V(T)$
there exists a unique path from $r$ to $x$. A \textit{directed
forest} is a family of disjoint directed trees. We say that a
vertex $y$ is \emph{below} vertex $x$ in a directed tree $T$ if
there exists a unique path from $x$ to $y$. In this paper we write
$\overrightarrow{xy}$ for a directed edge
   from $x$ to
  $y$.
\begin{defn} Let $D$ be a directed graph.
The vertices of the \emph{complex of directed trees} $\Delta(D)$
are oriented edges of $D$. The faces of $\Delta(D)$ are all directed
forests that are subgraphs of $D$.
\end{defn}
The investigation of complexes of directed trees was initiated by
D. Kozlov in \cite{Kozlov-Trees}. The complex of directed trees of
a graph $G$ is recognized in \cite{Ch-Jos} as a discrete Morse
complex of this graph (the authors treat graph as a
$1$-dimensional complex). Directed forests of $G$ correspond with
Morse matchings on $G$. Complexes of directed trees are also
studied in \cite{Eng} and \cite{Koz-chain}.

A $d$-dimensional simplicial complex is \textit{\textit{pure}} if
every simplex of dimension less than $d$ is a face of some
$d$-simplex. For further definitions about simplicial complexes
and other topological concepts used in this paper we refer the
reader to the textbook \cite{MunAT}.

\begin{defn}\label{D:shell}
A simplicial complex $\Delta$ is \textit{shellable} if $\Delta$ is
pure and there exists a linear ordering (\textit{\textit{shelling
order}}) $F_1, F_2,\ldots, F_k$ of maximal faces (\textit{facets})
of $\Delta$ such that for all $i < j\leqslant k$, there exist some
$l < j$ and a vertex $v$ of $F_j$, such that
\begin{equation}\label{e:definshell}
F_i\cap F_j \subseteq F_l\cap F_j
=F_j\setminus\{v\}.\end{equation}
\end{defn}
  For a fixed shelling order $F_1, F_2,\ldots, F_k$ of $\Delta$,
  the \textit{restriction} $\mathcal{R}(F_j)$ of the facet $F_j$ is
  defined by:
$$\mathcal{R}(F_j) = \{v \textrm{ is a vertex of }F_j :
F_j \setminus \{v\}\subset F_i\textrm{ for some }1 \leqslant i <
j\}.$$

 Geometrically, if we build up $\Delta$ from its facets
according to the shelling order, then $\mathcal{R}(F_j)$ is the unique
minimal new face added at the $j$-th step. The \textit{type} of
the facet $F_j$ in the given shelling order is the cardinality of
$\mathcal{R}(F_j)$, that is, $type(F_j) = |\mathcal{R}(F_j) |$.

\noindent  For a $d$-dimensional simplicial complex $\Delta$ we
denote the number of $i$-dimensional faces of $ \Delta $ by $f_i$,
and call $f(\Delta) = (f_{-1},f_0, f_1,\ldots , f_d)$  the
$f$-\textit{vector}. A new invariant, the $h$-\textit{vector} of
$d$-dimensional complex $\Delta$ is $h(\Delta) = (h_0, h_1,\ldots,
h_d, h_{d+1})$ defined by the formula
$$h_k=\sum_{i=0}^{k}(-1)^{k-i}{d+1-i\choose d+1-k}f_{i-1}.$$

\noindent If a simplicial complex $\Delta$ is shellable, then
$$h_k(\Delta)=|\{F\textrm{ is a facet of }\Delta: type(F)=k\}|$$ is
an important combinatorial interpretation of $h(\Delta)$. This
interpretation of the $h$-vector was of great significance in the
proof of the upper-bound theorem and in the characterization of
$f$-vectors of simplicial polytopes (see chapter 8 in
\cite{Zi-books}).

 If a $d$-dimensional simplicial complex $\Delta$ is
shellable, then $ \Delta$ is homotopy equivalent to a wedge of
$h_d$ spheres of dimension $d$. A set of maximal simplices from a
simplicial complex $\Delta$ is a set of \textit{generating
simplices} if the removal of their interiors makes $\Delta$
contractible.

For a given shelling order of a complex $\Delta$ we have that
$$\{F\in\Delta:F\textrm{ is a facet and } \mathcal{R}(F)=F\}$$ is a set of
generating facets of $\Delta$. Note that a facet $F$ is in this
set if and only if
\begin{equation}\label{genco}
\forall v \in F\textrm{ there exists a facet }
F'\textrm{ before }F\textrm{ such that }F\cap F'= F\setminus\{v\}.
\end{equation}

\noindent The concept of shellability for nonpure complexes is
introduced in \cite{BjWa}. In the definition of shellability of
nonpure complexes we just drop the requirement of purity from
Definition \ref{D:shell}.

For a facet $F$ of a shellable nonpure complex we can define its
restriction $\mathcal{R}(F)$ as before. For nonpure complexes the
definitions of $f$-vector and $h$-vector are extended for double indexed
arrays. For a nonpure complex $\Delta$ let
$$f_{i,j}(\Delta)=|\{A\in \Delta:|A|=j \textrm{, }
i= max\{|T|:A\subseteq T\subseteq\Delta\} \}|,$$
$$\textrm{and }
h_{i,j}(\Delta)=\sum_{k=0}^j(-1)^{j-k}{i-k\choose j-k}f_{i,k}.$$

The above defined arrays are called the $f$-\textit{triangle} and
the $h$-\textit{triangle} of $\Delta$. If $\Delta$ is a shellable
complex, we have the following combinatorial interpretation of the
$h$-triangle: $h_{i,j}(\Delta)=|\{F\textrm{ a facet of }\Delta:
|F|=i, |\mathcal{R}(F)|=j\}|.$

If a nonpure simplicial complex $\Delta$ is shellable, we know
that $\Delta$ has a homotopy type of the wedge of spheres,
consisting of $h_{j,j}$ copies of the $(j-1)$-spheres (see Theorem
4.1 in \cite{BjWa}). The conditions described in (\ref{genco})
help us to identify a generating set of a nonpure shellable
complex.\\ More information about shellable complexes can be found
in \cite{Bj:TM},
  \cite{BjSCMP} and
\cite{BjWa}.
\section{Shelling of graphs
 with a complete source}

  A vertex $x$ is a \textit{\textit{complete source}}
  of a directed graph $D$ if $\overrightarrow{xy} \in
  E(D)$ for all $y\in V(D)\setminus\{x\}$.
 D. Kozlov proved (Theorem 3.1 in \cite{Kozlov-Trees}) that if
 a directed graph $D$ has a complete source, then the complex
 $\Delta(D)$ is shellable. He used a version of shelling
 described in the following remark.
\begin{rem}\label{shellofcomplete} Let $\Gamma$ be a
 simplicial complex. Assume that we can partition all of
  the facets of $\Gamma$ into  sets
  $\mathcal{F}_0, \mathcal{F}_1, \mathcal{F}_2, \ldots, \mathcal{F}_{m}$
  such that the following holds:
  $$|\mathcal{F}_0|=1; \textrm{ for all } i\leqslant j,
  \textrm{ and for different facets } F\in \mathcal{F}_i,
  F'\in \mathcal{F}_j,$$

\vspace{-.75cm}

   \begin{equation}\label{shellpro}\textrm{ there exist }k<j,
   \textrm{ a facet } F''\in \mathcal{F}_k\textrm{, }
\textrm{ and a vertex } v\in F'
    \end{equation}

\vspace{-.55cm}

$$  \textrm{ such that }
F\cap F'\subseteq F''\cap F' =F'\setminus \{v\}.$$

\noindent In that case any linear order that refines the above
partition $\mathcal{F}_0, \mathcal{F}_1, \mathcal{F}_2, \ldots,
\mathcal{F}_{m}$ (for $i<j$ we list facets from $ \mathcal{F}_i$
before facets from $\mathcal{F}_j$)
 is a shelling of $\Gamma$.

 \noindent If $T$ is a directed tree and $v\in V(T)$ let $d_T(v)$ denote
 the outdegree of $v$, i.e.,
 $$d_T(v)=|\{x\in V(T):\overrightarrow{vx}\in E(T)\}|.$$
 In the proof of Theorem 3.1 in \cite{Kozlov-Trees}, the facets
 of $\Delta(D)$ are ordered by their degree sequences,
  i.e., trees $T$ and $T'$ are in the same class if and
  only if $d_T(v)=d_{T'}(v)$ for all $v\in V(D)$. Substantially,
  the facets of $\Delta(D)$ are classified by considering the out-degree
  of the complete source.
\end{rem}

Here we consider a directed graph $D$ with a complete source $c$
 and detect some nice properties of a
shelling
 described in the above remark.\\ If $|V(D)|=n$ for $i=0,1,\ldots,n-1$,
 we set
$$\mathcal{F}_i=\{T\textrm{ a facet in }\Delta(D): d_T(c)=n-i-1\}.$$
In the same manner as in the proof of Theorem 3.1 in
\cite{Kozlov-Trees} we can verify that the partition
$\mathcal{F}_0, \mathcal{F}_1, \mathcal{F}_2, \ldots,
\mathcal{F}_{n-1}$ fulfills the condition described in Remark
\ref{shellofcomplete}. Namely, if $d_T(c)\geqslant d_{T'}(c)$ and
$T\neq T'$, then there exists an edge $\overrightarrow{xy}\in
T'\setminus T$ such that $x\neq c$. We define
$$T''=T' \setminus\{\overrightarrow{xy}\}
\cup\{\overrightarrow{cy}\} \textrm{ if the vertex }
 c\textrm{ is not below } y \textrm{ in } T'
  \textrm{;\hspace{.3cm} or }$$
$$T''=T' \setminus\{\overrightarrow{xy}\}
 \cup\{\overrightarrow{cr}\} \textrm{ if } c\textrm{ is below } y
  \textrm{ in } T' \textrm{ and } r\textrm{ is the root of }T'.$$
In both cases simplices $T,T',T''$ and the vertex
$\overrightarrow{xy}$ satisfy condition described in
(\ref{shellpro}).

\noindent Furthermore, for a facet $T\in\Delta(D) $ the unique new
face for $T$ in the shelling order defined above is
$\mathcal{R}(T)= \{\overrightarrow{xy}\in T: x\neq c\}$.
Therefore, the type of $T$ is $type(T)=n-1-d_T(c)$, and we obtain
that
$$h_i(\Delta(D))=|\mathcal{F}_i|=
|\{T\textrm{ is a facet of }\Delta(D) :d_T(c)=n-i-1\}|. $$
\begin{cor} Let $G_n$ be the complete directed graph on $n$
vertices. Then, for all $k=0,1,\ldots,n-1$ we have
$$h_k(\Delta(G_n))={n-1\choose k}(n-1)^k.$$
\end{cor}

\begin{rem}\label{R:propertiesofshell}
If a directed graph $D$ has a complete source, then the shelling of
$\Delta(D)$ is straightforward in the following sense:
\begin{itemize}
  \item [(1)] We start the shelling with an appropriate facet $F_0$
  and let $\mathcal{F}_0=\{F_0\}$.
\item [(2)] When we order all of the facets from
$\mathcal{F}_{i-1}$, let $\mathcal{F}_{i}$ denote the set of all
 facets of $\Delta(D)\setminus
 (\mathcal{F}_{0} \cup\cdots\cup\mathcal{F}_{i-1})$ that are
 neighborly (share a common ridget)
 to a
simplex from $\mathcal{F}_{i-1}$. \item[(3)] We continue
shelling of $\Delta(D)$ by arranging simplices from
$\mathcal{F}_i$ in an arbitrary order.
\end{itemize}

\begin{itemize}
\item [(4)] In this shelling order, for any facet $F$ we have that
$type(F)=i \Leftrightarrow F\in \mathcal{F}_i .$
\end{itemize}
\end{rem}
It may be interesting to find more examples of simplicial
complexes that allow a shelling with the properties $(1)$--$(4)$ from
the above remark.
\begin{exm}
Let $D_n$ be the directed graph with $V(D_n)=[n]$ and
$$E(D_n)=\{\overrightarrow{1i}:i\in [n], i\neq 1\}\cup
 \{\overrightarrow{2j}:j\in[n], j\neq 2\}.$$
 It  is easy to see that $\Delta(D_n)$ is combinatorially
 equivalent to
 the boundary of
 the $(n-1)$-dimensional crosspolytope.
\end{exm}
\begin{thm}
The only simplicial $d$-dimensional polytope whose boundary admit
a shelling as those described in Remark 4 is the crosspolytope.
\end{thm}
\begin{proof}
Assume that $P$ is a simplicial $d$-polytope with desired
shelling. We identify a facet of $P$ with its set of vertices. Let
$F_0=\{v_1,v_2,\ldots,v_d\}$ be the first facet in this shelling.

 Let $w_i$ denote the unique new vertex of the facet of
$P$ that contains $(d-2)$-dimensional simplex
$F_0\setminus\{v_i\}$. All of the facets of $P$ whose type is $1$
belong to $\mathcal{F}_1$ and therefore have the form
$F_0\setminus\{v_i\}\cup \{w_i\}$. We can conclude that the
set of the vertices of $P$ is
 $V(P)=\{v_1,v_2,\ldots,v_d,w_1,w_2,\ldots,w_d\}.$

For any $S\subseteq [d]$ we consider the $(d-1)$-simplex
$$F_S=conv\left(\{v_i:i\notin S\}\cup\{w_j:j\in S\}\right).$$
We do not know that $F_S$ is a facet of $P$, but we use
 induction on $k$ to show that
 \begin{equation}\label{F_i}
 \mathcal{F}_k=\{F_S: S\subseteq [d], |S|=k\}.
 \end{equation}
Assume that the above statement holds for all $t\leqslant k-1$.
Let $F\in \mathcal{F}_{k}$ be a facet (yet not listed) of $P$ that
shares a common ridget with a facet $\bar{F}$ from
$\mathcal{F}_{k-1}$.

 From the inductive hypothesis we have $\bar{F}=F_S$
 (for $S\subset[n], |S|=k-1$) and
 $F=F_S\setminus \{v_i\} \cup \{w_j\}$ for $i,j \notin S$.
 If $i\neq j$ then the edge $\{v_jw_j\}$ and the
 $(k-1)$-simplex $\{w_s:s\in S\cup\{j\}\}$
 are two different minimal
 new faces that $F$ contributes in the shelling of $P$, which is
  impossible.
Therefore, we can conclude that $i=j$, and
$F=F_S\setminus\{v_i\}\cup\{w_i\}=F_{S\cup\{i\}}.$

\noindent We have that any of the facets that belong to
$\mathcal{F}_{k}$ has the form described in (\ref{F_i}). All of
the facets from $\mathcal{F}_{k}$ can be
 listed in an arbitrary order and any of them has the type $k$.
  Therefore, we conclude that two facets from $\mathcal{F}_k$
   cannot share the same ridget, and we obtain that
    $$(d-k+1)|\mathcal{F}_{k-1}|=k|\mathcal{F}_{k}|.$$
    The inductive assumption and the above equations
complete the proof of (\ref{F_i}). So, we may conclude that
$P$ is
combinatorially equivalent with $d$-dimensional crosspolytope.

\end{proof}
If a directed graph $D$ has a complete source $c$ then the complex
$\Delta(D)$ is homotopy equivalent to a wedge of the spheres. In
\cite{Kozlov-Trees}, D. Kozlov describes generating facets of
 $\Delta(D)$ as rooted trees of $D$ having complete source $c$ as a leaf.

  Here we study the combinatorics of the spheres in $\Delta(D)$ when
  $D$ has a complete source.
  For each tree $T$ that is a generating facet we associate
  a sphere
  $S_T\subset \Delta(D)$
  that contains $T$ and describe the combinatorial type of $S_T$.

 We consider a directed graph $D$ with $n$ vertices. Assume that $c$
is a complete source of $D$. Let $T$ be a rooted spanning tree of
$D$ with vertex $c$ as a leaf. If $x_1\rightarrow
x_2\rightarrow\ldots\rightarrow x_{k}\rightarrow c$ is the unique
directed path from $x_1$ (the root of $T$) to $c$, let $\sigma_T$
denote the simplex
$\{\overrightarrow{x_1x_2},\overrightarrow{x_2x_3},\ldots,
\overrightarrow{x_{k}c},\overrightarrow{cx_{1}}\}$. It is obvious
that $\sigma_T\notin \Delta(D)$. Also note that
$\partial\sigma_T\subset \Delta(D)$.

\noindent Let
$A=\{y_1,y_2,\ldots,y_{r}\}=V(D)\setminus\{x_1,x_2,\ldots,x_k,c\}$,
i.e., $A$ contains $r=n-k-1$ vertices that do not belong to the
unique directed path from $x_1$ to $c$ in $T$. For any $y_i\in A$
there exists the unique vertex $z_i$ such that
$\overrightarrow{z_iy_i}\in E(T)$. Now, we define
\begin{equation}\label {E:SferaS_T}
S_T=\partial\sigma_T*\{\overrightarrow{z_1y_1},
\overrightarrow{cy_1}\}
*\{\overrightarrow{z_2y_2},\overrightarrow{cy_2}\}*\cdots*
\{\overrightarrow{z_{r}y_{r}},\overrightarrow{cy_{r}}\}.
\end{equation}
 It is
not complicated to prove that $S_T\subset \Delta(D)$. The sphere
$S_T$ is $(n-k-1)$-folded bipyramid over the boundary of
$k$-simplex $\sigma_T$.

\begin{prop}
If a directed graph $D$ has two complete sources, then $\Delta(D)$
is the union of the spheres defined in (\ref{E:SferaS_T}).
\end{prop}
\begin{proof}
Let us denote two complete sources in $D$ by $c$ and $c'$. If $c$
is a leaf in $T$, then we have $T\in S_T$. If $c$ is not a leaf in
a tree $T$, then let $\{x_1,x_2,\ldots,x_k\}$ be the set of all
vertices of $D$ such that $\overrightarrow{cx}_i\in E(T)$ for all
$i=1,2,\ldots,k$.

If the vertex $c'$ is not below $c$ in $T$, we define
$$T'=T\setminus
\{\overrightarrow{cx}_1,\overrightarrow{cx}_2,\ldots,
\overrightarrow{cx}_k\}\cup
\{\overrightarrow{c'x}_1,\overrightarrow{c'x}_2,\ldots,
\overrightarrow{c'x}_k\}.$$
In the case when $c'$ is below $c$ (then we have that $c'=x_i$ or
$c'$ is below $x_i$) and the root of $T$ is $r$ we define
$$T'=T\setminus
\{\overrightarrow{cx}_1,\overrightarrow{cx}_2,\ldots,
\overrightarrow{cx}_k\} \cup
\{\overrightarrow{c'x}_1,\ldots,\overrightarrow{c'x}_{i-1},
\overrightarrow{c'r} ,\overrightarrow{c'x}_{i+1},
\ldots,\overrightarrow{c'x}_k\}.$$ In both cases the directed tree
$T'$ is a generating facet of $\Delta(D)$. Obviously, the facet
$T$ is contained in the sphere $S_{T'}$.

\end{proof}
We conclude now that $\Delta(G_n)$ is the union of the
$(n-k-1)$-folded bipyramids over the boundary of $k$-simplex. A
simple calculation and the well-known formulae for the number of
forests with $n-1$ vertices and $k$ trees such that $k$ specified
nodes belong to distinct trees (Theorem 3.3 in \cite{moon}) give
us the number of spheres in $\Delta(G_n)$ of the same
combinatorial type.
\begin{cor}For any $n\geqslant 1$ the complex $\Delta(G_n)$
is a union of $(n-1)^{n-1}$ spheres of dimension $n-2$. For
$0<k<n$ there are exactly
 $$\frac{(n-1)!}{(n-k-1)!}k (n-1)^{n-k-2}$$
 of these spheres that are
 $(n-k-1)$-folded bipyramid over the boundary of $(k-1)$-simplex.

 \end{cor}

\section{Shellability of skeleton of $\Delta(D)$}
 The subcomplex of a complex of directed trees
generated by its maximal facets was studied in \cite{Aya} and
\cite{Ch-Jos}.

Here we ask about the minimal dimension of the facets of
$\Delta(D)$, i.e., we want to determine the maximal $k$ such that
the $k$-skeleton of $\Delta(D)$ is pure. Note that for any directed
graph $D$ we have that the $k$-skeleton of $\Delta(D)$ is
$$\Delta^{(k)}(D)=\{F: F \textrm{ is a rooted forest in } D
\textrm{ with at least } |V(D)|-k-1\textrm{ trees }\}.$$ For a
simple graph $G$ let $\overrightarrow{G}$ denote the directed
graph obtained by replacing every edge $xy$ of $G$ with two
directed edges $\overrightarrow{xy}$ and $\overrightarrow{yx}$.

\noindent The greatest distance between two vertices of a graph
$G$ is the \textit{diameter} of $G$, denoted by $diam(G)$. A
subset of the vertex set of a graph is \textit{independent} if no
two of its elements are adjacent. The set of neighbors of a vertex
$v$ in a graph $G$ is denoted by $N(v)$.

For a graph $G$ we say that $A\subseteq V(G)$ is a
\textit{strongly independent} set if $A$ is independent and
$N(u)\cap N(v)=\emptyset$ for all $u,v \in A, u\neq v$. Let $r(G)$
denote the maximal cardinality of a strongly independent subset
of $V(G)$.

\begin{prop}\label{P:ostrogonezavisnim}
The $k$-skeleton of $\Delta(\overrightarrow{G})$ is pure if and
only if $k\leqslant |V(G)|-1-r(G)$.
\end{prop}
\begin{proof}
Let $F$ be a directed forest of $\overrightarrow{G}$ with roots
$x_1,x_2,\ldots,x_t$. If the forest $F$ is a facet of
$\Delta(\overrightarrow{G})$, then $\{x_1,x_2,\ldots,x_t\}$ is an
independent set in $G$. Further, if $T_i$ denotes the tree of $F$
that contain $x_i$, then $N(x_i)\subseteq V(T_i)$. Therefore we
obtain that $\{x_1,x_2,\ldots,x_t\}$ is a strongly independent
set.

So, minimal facets of $\Delta(\overrightarrow{G})$ correspond with
maximal strongly independent sets of $G$.

\end{proof}
\begin{cor}
For a connected graph $G$, the complex $\Delta(\overrightarrow{G})$
is pure if and only if $diam(G)\leqslant 2$.
\end{cor}
\noindent For a graph $G$ let $\textrm{m}_G$ denote the maximal
dimension of skeleton of $\Delta(\overrightarrow{G})$ that is
pure. From Proposition \ref{P:ostrogonezavisnim} we know that
$\textrm{m}_G=|V(G)|-r(G)-1$. Now we examine shellability of
$\Delta^{(\textrm{m}_G)}(\overrightarrow{G})$.

We say that a maximal strongly independent set
$A=\{x_1,x_2,\ldots,x_r\}$ of a graph $G$ is a \textit{complete
$r$-source} if $V(G)=A\cup N(x_1)\cup N(x_2)\cdots \cup N(x_r)$.

\begin{thm}\label{shelofskel}
If a graph $G$ has a complete $r$-source, then
$\Delta^{(\textrm{m}_G)}(\overrightarrow{G})$ is shellable.
\end{thm}
\begin{proof}
Let $A=\{x_1,x_2,\ldots,x_r\}$ be a complete $r$-source in $G$.
Assume that the vertex set $V(G)$ is linearly ordered. For a facet
$F$ of
 $\Delta^{(\textrm{m}_G)}(\overrightarrow{G})$
  (recall that $F$ is a directed forest with $r$ trees) we define
 $d_F=(d_F(x_1),d_F(x_2),\ldots,d_F(x_r))$ and $
 S_F=(F_1,F_2,\ldots,F_r)$
  where $F_i=\{v\in N(x_i):\overrightarrow{x_iv}\in E(F)\}$.

Let $<_L$ denote the lexicographical order on $\mathbb{N}^r$. We say
that $S_F \preceq S_{F'}$ if and only if $F_1=F'_1,\ldots,
  F_{i-1}=F'_{i-1}$ and $min (F_i\triangle F'_i )\in F_i$.
Now, we define a partial order on the facets of
$\Delta^{(\textrm{m}_G)}(\overrightarrow{G})$:
$$F <F' \Leftrightarrow \left\{%
\begin{array}{ll}
    d_F' <_L d_F, & \hbox{or ;} \\
   d_F' = d_F\textrm{ and }S_F \preceq S_{F'}. & \hbox{} \\
\end{array}%
\right.    $$ The above order induces a partition of the facets of
$\Delta^{(\textrm{m}_G)}(\overrightarrow{G})$. A block in this
partition contains all forests of
$\Delta^{(\textrm{m}_G)}(\overrightarrow{G})$ in which the sets of
outgoing edges having $x_i$ as the source are the same for all
$i=1,2,\ldots,r$. Note that the relation $<$ induces a linear
order on the blocks. The forest with edges
$\{\overrightarrow{x_iv}:x_i\in A, v\in N(x_i)\}$ is the only
facet contained in the first block.

 Now, we will prove that this partition of the facets of
$\Delta^{(\textrm{m}_G)}(\overrightarrow{G})$ satisfies conditions
described in Remark \ref{shellofcomplete}.
 Consider two different forests $F,F'\in \Delta^{(\textrm{m}_G)}
 (\overrightarrow{G})$ such that
 $F\in \mathcal{F}_i, F'\in \mathcal{F}_j$ and $i\leqslant j$.
Let $T_1,T_2,\ldots,T_s$ denote the trees of the forest $F\cap
F'$. For $i=1,2,\ldots,s$ let $r_i$ denote the root of $T_i$.
Note that $s>r$ and all edges from $E(F)\setminus E(F')$ have
the form $\overrightarrow{xr_i}$. We consider the following three
possibilities:
 \begin{itemize}
    \item[1.] There exists an edge
    $\overrightarrow{uv}\in E(F')\setminus E(F)$ such that
    $u\notin A$. As we have that $r<s$, we can conclude that there
    exists $j$ such that $r_j\in N(x_i)$ and $x_i$ is not below
    $r_j$ in $F'\setminus \{\overrightarrow{uv}\}$. Then we set $ F''=F'\setminus
\{\overrightarrow{uv}\}\cup \{\overrightarrow{x_ir_j}\}$.
 \end{itemize}

Now, we assume that $\overrightarrow{uv}\in E(F')\setminus
    E(F)$ implies $u\in A$. Further, let $i_0$ denotes the minimal
$i\in [r]$
    for which there exists an edge $\overrightarrow{x_{i_0}u}\in E(F')
    \setminus
    E(F)$.
     \begin{itemize}
    \item[2.]
    If $E(F')\setminus E(F)$ also contains an edge
    $\overrightarrow{x_j v}$ such that $i_0<j$, then from
    $d_F(x_{i_0})\geqslant d_{F'}(x_{i_0})$ we conclude that there
    exists $\overrightarrow{x_{i_0}z}\in E(F)\setminus
    E(F')$. The vertex $z$ is the root in $F'$ and $x_{i_0}$ is not below
    $z$ in $F'$.
    Otherwise we have an
    edge $\overrightarrow{xz}$ (or $\overrightarrow{zw}$) in $E(F')\setminus E(F)$,
    such that $x\notin A$ (or $z\notin A$) . In this case we set $F''=F'\setminus
\{\overrightarrow{x_jv}\}\cup \{\overrightarrow{x_{i_0}z}\}.$
    \item[3.] If $E(F')\setminus E(F) =
    \{\overrightarrow{x_{i_0}v_1},\overrightarrow{x_{i_0}v_2},
    \ldots, \overrightarrow{x_{i_0}v_m}\}$, we have that there
    exists the edge $\overrightarrow{x_{i_0}u}\in E(F)\setminus E(F')$
    such that $u$ is smaller than any of $v_i$ in the linear order
    defined on $V(G)$. Again $u$ is the root in $F'$, and we set
    $F''=F'\setminus
\{\overrightarrow{x_{i_0}v_1}\}\cup \{\overrightarrow{x_{i_0}u}\}.$
 \end{itemize}

 In any of the cases considered above, it is clear that the forests $F,F', F''$
 satisfy (\ref{shellpro}).

\end{proof}

Now, we investigate shellability of
$\Delta^{(\textrm{m}_{C_n})}(\overrightarrow{C}_n)$, where $C_n$
denotes a cycle with $n$ vertices.
\begin{thm}
A complex $\Delta^{(\textrm{m}_{C_n})}(\overrightarrow{C}_n)$ is
shellable if and only if  $n=3k$ or $n=3k+1$. 

\end{thm}
\begin{proof} Note that $r(C_n)=\lfloor\frac{n}{3}\rfloor$
and therefore we have that
$$\textrm{m}_{{C_n}}=\left\{
               \begin{array}{ll}
                 2k-1, & \hbox{if $n=3k$;} \\
                 2k, & \hbox{if $n=3k+1$;} \\
                 2k+1, & \hbox{if $n=3k+2$.}
               \end{array}
             \right.
$$

Let $\mathcal{C}_n$ denote the
 simplicial complex with $n$ vertices indexed by $\mathbb{Z}_n$
 and $F\subseteq \mathbb{Z}_n$ is a face if and only if it does not
 contain $\{i,i+1\}$ for $i\in \mathbb{Z}_n$. It is obvious that
  $ \Delta(\overrightarrow{C}_n)=
 \mathcal{C}_{2n}\setminus\left\{[1,3,\ldots,2n-1],[2,4,\ldots,2n]\right\}$
 and therefore
 $ \Delta^{(\textrm{m}_{C_n})}(\overrightarrow{C}_n)=
 \mathcal{C}_{2n}^{(\textrm{m}_{C_n})}.$

If $n=3k$, then $\{1,4,7,\ldots,3k-2\}$ is a complete $r$-source
for $C_n$ and from Theorem \ref{shelofskel} we know that
$\Delta^{(2k-1)}(\overrightarrow{C}_{3k})$ is shellable.

 If $n=3k+1$, we will prove that the lexicographical order
of the facets of $\mathcal{C}_{6k+2}^{(2k)}$ defined  by $A<B$
if and only if $min(A\vartriangle B)\in A$ is a shelling order.\\
For
 $A=\{a_0,a_1,\ldots, a_{2k}\}<B=\{b_0,b_1,\ldots,b_{2k}\}$ let
 $a_i=min(A\vartriangle B)\in A $ and
 let $b_j=min\textrm{ }B\setminus A$.
\noindent  We consider $C=(B\setminus\{b_j\})\cup\{a_i\}$. Note
that $C$ is not contained in $\mathcal{C}_{6k+2}^{(2k)}$ if and
only if $a_0=1$, $b_0=2$, $b_{2k}=6k+2$. In that case, because
$1\in A$ we have that $6k+2\notin A$.

 If $b_{2k}=6k+2$ and $b_{2k-1}< 6k$,
 then we define $C=B\setminus\{6k+2\}\cup\{6k+1\}$.

\noindent If $b_0=2$, $b_{2k-1}=6k$, $b_{2k}=6k+2$,
 then there exists $s\in\{1,2,\ldots, 2k-1\}$ such that
 $b_s-b_{s-1}>3$.
  Then, we let $C=B\setminus\{6k+2\}\cup\{b_{s-1}+2\}$.
It is easy to check that the condition described in
(\ref{e:definshell}) is satified in any of the above cases. So, we
can conclude that $\Delta^{(2k)}(\overrightarrow{C}_{3k+1})$ is
shellable.

\noindent For $n=3k+2$ we consider complex
 $\Delta^{(2k+1)}(\overrightarrow{C}_{3k+2})=
 \mathcal{C}_{6k+4}^{(2k+1)}$.
 We know that $\mathcal{C}_{6k+4}$ is homotopy
equivalent with a $2k$-dimensional sphere (see Proposition 5.1 in
\cite{Kozlov-Trees}). From the proof of this proposition we can
identify this sphere with the boundary of $(2k+1)$-dimensional
crosspolytope $\{1,2\}*\{4,5\}*\cdots*\{6k+1,6k+2\}.$
\\
Obviously, this sphere is contained in
$\mathcal{C}_{6k+4}^{(2k+1)}$.

 However, $\mathcal{C}_{6k+4}^{(2k+1)}$ also
contains $(2k+1)$-dimensional spheres (boundaries of
$(2k+2)$-simplex $\{1,3,5,\ldots,4k+5\}$ in $\mathcal{C}_{6k+4}$).

  Therefore,
   we obtain that
  this complex is homotopy equivalent to a wedge of
  spheres of different dimensions.
 So, we conclude that $\Delta^{(2k+1)}(\overrightarrow{C}_{3k+2})$ is not
 shellable.
 \end{proof}

\section{Trees}

 For a simple graph $G=(V,E)$ the \textit{independency complex} $I(G)$ is
the simplicial complex with vertex set $V$ and with faces the
independent sets of $G$. The independence complex has been
previously studied in \cite{Eh_He},\cite{Meshulam}.

Shellability and vertex-decomposability of independency complexes
is disscussed in\cite{Doh-Eng} and \cite{Wood}. A complex $\Delta$
is \emph{vertex decomposable} if it is a simplex or (recursively)
$\Delta$ has a shedding vertex $v$ such that $\Delta \setminus
\{v\}$ and $link_{\Delta} v$ are vertex decomposable. It is
well-known that any vertex decomposable complex is shellable too.

A \emph{chordless} cycle of length $n$ in a graph $G$ is a cycle
$v_1, v_2, \ldots, v_n,v_1$ in $G$ with no chord, i.e. with no
edges except $\{v_1v_2,v_2v_3,\ldots,v_{n-1}v_n,v_nv_1\}$.

We will use the following Theorem.
\begin{thm}[\textrm{Theorem 1, }\cite{Wood}]\label{Russ}
 If $G$ is a graph with no chordless cycles of length other than 3
or 5, then $I(G)$ is vertex decomposable (hence shellable and
sequentially Cohen-Macaulay.)
\end{thm}

 We follow Kozlov \cite{Kozlov-Trees} and
say that a digraph $D$ is essentially a tree if it becomes an
undirected tree when one replaces all directed edges (or pairs of
directed edges going in opposite directions) by an edge.

\begin{thm}
Let $D=(V(D),E(D))$ be essentially a tree. Then $\Delta(D)$ is
vertex decomposable and hence shellable.
\end{thm}
\begin{proof}
For a given tree $D$ we define a simple graph $G$ in the following
way.
 For $v\in
V(D)$ let $d^-(v)=|\{x\in V(D):\overrightarrow{xv} \in E(D)\}|$
denote the in-degree of $v$ in $D$. We replace every $v\in V(D)$
with a complete graph $K_{d^-(v)}$ whose vertices correspond with
 directed edges
having $v$ as sink. Further, if both of directed edges
$\overrightarrow{uv}, \overrightarrow{vu}$ are
 contained in $E(D)$, then the
corresponding vertices of $K_{d^-(v)}$ and $ K_{d^-(u)}$ are
adjacent in $G$. Formally, we define $V(G)= E(D)$, and edges with
the same sink $\overrightarrow{ax},\overrightarrow{bx}$ are
adjacent in $G$. Also, if $\overrightarrow{ab},
\overrightarrow{ba} \in E(D)$ they are adjacent as vertices of
$G$.

Note that $A\subset V(G)$ is an independent set in $G$ if and only
if $A$ is the set of edges of a directed forest in $D$.
 Therefore we have that $\Delta(D)=I(G)$.
 Moreover, the construction
of $G$ and the assumption that $D$ is essentially a tree guaranteed
that $G$ does not contain a chordless cycle of length other than
$3$. Now, the statement of our theorem follows from Theorem
\ref{Russ}.

\end{proof}
We describe a way to find an explicit shelling of $\Delta(D)$. Let
$D$ be a directed graph and let $v\in V(D)$ be a leaf in $D$. In
other words there exists the unique vertex $x\in V(D)$ such that
$\overrightarrow{vx}$ or $\overrightarrow{xv}$ or both of them are
in $E(D)$ and there are no other edges where $v$ is a source or a
sink.

 Let $D'=D\setminus\{v\}$ and let
$\{y_1,y_2,\ldots,y_k\}= \{y\in V(D'):\overrightarrow{yx}\in
E(D')\}$. Furthermore, let
$D_0=D'\setminus\{\overrightarrow{y_1x},\overrightarrow{y_2x},
\ldots, \overrightarrow{y_kx}\}$ and assume that
$\overrightarrow{xy_i}\in E(D)$ for $i=1,2,\ldots,s$. Now, for
$p=1,2,\ldots,k$ we set
$D_p=D_0\setminus\{\overrightarrow{xy_i}\}$. Note that $D_p=D_0$
for $p>s$.

\noindent We know that the complexes $\Delta(D')$, $\Delta(D_0)$
and $\Delta(D_p)$ are shellable. \\Assume that: \begin{description}
    \item[(i)] $F_1,F_2,\ldots,F_t$ is a shelling of
    $\Delta(D')$;
    \item[(ii)] $H_1,H_2,\ldots,H_s$ is a shelling of
    $\Delta(D_0)$;
    \item[(iii)] $G^p_1,G^p_2,\ldots,G^p_{t_p}$ is a shelling of
    $\Delta(D_p)$ (for $p=1,2,\ldots,k$).
\end{description}
We use the above notation in the next proposition.
\begin{prop}\label{P:recursiveshelling}
We consider three possible cases.
\begin{itemize}
    \item[(a)] If $\overrightarrow{xv}\in E(D)$ and
    $\overrightarrow{vx}\notin E(D)$, then
    $F_1\cup\{\overrightarrow{xv}\},
    F_2\cup\{\overrightarrow{xv}\},
\ldots,F_t\cup\{\overrightarrow{xv}\}$ is a shelling of
$\Delta(D)$. Also, we have that
$h_{i,j}(\Delta(D))=h_{i-1,j}(\Delta(D'))$.
    \item[(b)] If $\overrightarrow{xv}\notin E(D)$ and
    $\overrightarrow{vx}\in E(D)$, then
    $$H_1\cup\{\overrightarrow{vx}\},
\ldots,H_s\cup\{\overrightarrow{vx}\},
G^1_1\cup\{\overrightarrow{y_1x}\},\ldots,
G^1_{t_1}\cup\{\overrightarrow{y_1x}\}, \ldots,G^k_{t_k}
 \cup\{\overrightarrow{y_kx}\}$$
is a shelling of $\Delta(D)$. Furthermore, we have that
$$h_{i,j}(\Delta(D))=h_{i-1,j}(\Delta(D_0))+
\sum_{p=1}^kh_{i-1,j-1}(\Delta(D_p)).$$
    \item[(c)] If $\overrightarrow{xv},
    \overrightarrow{vx}\in E(D)$, then
    $$F_1\cup\{\overrightarrow{xv}\},F_2
    \cup\{\overrightarrow{xv}\},
\ldots,F_t\cup\{\overrightarrow{xv}\},
H_1\cup\{\overrightarrow{vx}\},
H_2\cup\{\overrightarrow{vx}\},
\ldots,H_s\cup\{\overrightarrow{vx}\}$$ is a shelling of $D$.
In
that case we have that
$$h_{i,j}(\Delta(D))=h_{i-1,j}(\Delta(D'))+
h_{i-1,j-1}(\Delta(D_0)).$$
\end{itemize}
\end{prop}
\begin{proof}$ $

\begin{itemize}
\item[(a)] This is obvious, because $\Delta(D)$ is a cone over
$\Delta(D')$ with apex $\overrightarrow{xv}$. Therefore, we have
$\mathcal{R}_D(F_i\cup\{\overrightarrow{xv}\})
    =\mathcal{R}_{D'}(F_i)$ and $\Delta(D)$ is contractible.
    \item[(b)] If a facet $F$ of $\Delta (D)$ contains
    $\overrightarrow{vx}$, then $F$ does not contain any of edges
    $\{\overrightarrow{y_1x},\overrightarrow{y_2x},
\ldots, \overrightarrow{y_kx}\}$. So, in that case we have that
$F=H\cup \{\overrightarrow{vx}\}$, for a facet $H$ of
$\Delta(D_0)$. If a facet $F'$ of $\Delta (D)$ does not contain
    $\overrightarrow{vx}$, then $F'$ must contain exactly one of
  the edges $\{\overrightarrow{y_1x},\overrightarrow{y_2x},
\ldots, \overrightarrow{y_kx}\}$. Therefore, we have that
$F'=G\cup \{\overrightarrow{y_px}\}$ for a facet $G$ of $D_p$.

The supposed shelling of $\Delta(D_0)$ provides that for
$i\leqslant j$ and facets $H_i, H_j$ of $\Delta(D_0)$ there exists
$k\leqslant j$ and $\overrightarrow{wz}\in H_j$ such that
$$(H_i\cup\{\overrightarrow{vx}\})
\cap(H_j\cup\{\overrightarrow{vx}\})
\subseteq(H_k\cup\{\overrightarrow{vx}\})
\cap(H_j\cup\{\overrightarrow{vx}\}) = H_j
\cup\{\overrightarrow{vx}\}\setminus\{\overrightarrow{wz}\}.$$
    Note that for any $p$ such that $1\leqslant p\leqslant s$ and
    for any facet $G^p_j$ of $\Delta(D_p)$
    there exists a facet $H_k$ of $\Delta(D_0)$
    such that $G^p_j\subseteq
    H_k$. Therefore, for any facet $H_i$ of $\Delta(D_0)$
    we have
    $$(H_i\cup\{\overrightarrow{vx}\})\cap
     (G^p_j\cup\{\overrightarrow{y_px}\})\subseteq
     (H_k\cup\{\overrightarrow{vx}\})\cap
     (G^p_j\cup\{\overrightarrow{y_px}\})=G^p_j.$$
     Also, for $q\leqslant p$ and a facet $G^q_r$ of
     $\Delta(D_q)$ we have
    $$(G^q_r\cup\{\overrightarrow{y_qx}\})\cap
     (G^p_j\cup\{\overrightarrow{y_px}\})\subseteq
     (H_k\cup\{\overrightarrow{vx}\})\cap
     (G^p_j\cup\{\overrightarrow{y_px}\})=G^p_j.$$
So, we obtain that the order defined in (b) is a shelling order
for $\Delta(D)$. In this order we have that the restriction of the
facets of $\Delta(D)$ is
$$\mathcal{R}_D(H_i\cup\{\overrightarrow{vx}\})=
\mathcal{R}_{D_0}(H_i)\textrm{ and }
\mathcal{R}_D(G^p_i\cup\{\overrightarrow{y_px}\})=
\mathcal{R}_{D_p}(G^p_i)\cup \{\overrightarrow{y_px}\}.$$
    \item[(c)]

In this case a facet of $\Delta(D)$ has the form
$$\{\overrightarrow{xv}\}\cup F\textrm{, for a facet }F \textrm{ of }
\Delta(D')\textrm{ or } \{\overrightarrow{vx}\}\cup H\textrm{, for
a facet }H \textrm{ of }\Delta(D_0).$$

    Again, for a facet $H_j$ of $\Delta(D_0)$ there exists a facet $F_i$ of
    $\Delta(D')$ such that
    $H_j\subseteq F_i$. In the similar manner as in (b) we can
    prove that the considered order is a shelling order. Further,
     the restriction in this order is
     $$\mathcal{R}_D(F_i\cup\{\overrightarrow{xv}\})
    =\mathcal{R}_{D'}(F_i) \textrm{ and }
     \mathcal{R}_D(H_i\cup\{\overrightarrow{vx}\})=
\mathcal{R}_{D_0}(H_i)\cup \{\overrightarrow{vx}\}.$$
\end{itemize}
\end{proof}
\begin{rem}\label{R:generating}
Now, we can identify generating facets of $\Delta(D)$.
    If $\overrightarrow{xv}\in E(D)$ and
    $\overrightarrow{vx}\notin E(D)$, then $\Delta(D)$ is
    contractible.
  If $\overrightarrow{xv}\notin E(D)$ and
    $\overrightarrow{vx}\in E(D)$, let $\mathcal{G}_p$ denote a
    set
    of generating faces of $\Delta(D_p)$ for $p=1,2,\ldots,s$. Then,
    a generating set of facets of $\Delta(D)$
    is
    $$\bigcup_{p=1}^s
    \{G\cup\{\overrightarrow{y_px}\}:G\in \mathcal{G}_p\}.$$
If $\overrightarrow{xv},\overrightarrow{vx}\in
     E(D)$, then a set of generating facets of $\Delta(D)$ is
     $$\{H\cup\{\overrightarrow{vx}\}: H \textrm{ is a generating facet
     of
     }\Delta(D_0)\}.$$

\end{rem}
A \emph{directed acyclic graph} is a directed graph without
directed cycles. By successive applications of Proposition
\ref{P:recursiveshelling} and Remark \ref{R:generating} we obtain
the following result of A. Engstr\" om.
\begin{thm}[Theorem 2.10,  \cite{Eng}]
If $D$ is a directed acyclic graph, then $\Delta(G)$ is homotopy
equivalent to a wedge of $\prod_{v\in V(G)\setminus R}(d^-(v) -
1)$ spheres of dimension $|V (G)|-|R|- 1$, where $R$ is the set of
vertices without edges directed to them.
\end{thm}
\noindent Now, we investigate homotopy type of $\Delta(D)$ when
$D$ is a double directed tree.
\begin{defn}
A tree $T$ with $2n$ vertices ($n$ leaves and $n$ non-leaves) such
that every non-leaf is adjacent to exactly one leaf we call
\emph{basic tree}. Also, we say that a tree with exactly two
vertices is a basic tree. We say that the edge connecting a
non-leaf and a leaf is peripheral.
\end{defn}
 We can produce a basic tree if we start with an arbitrary
tree $T'$ and add a leaf to each node of $T'$. We use description
of generating facets from Remark \ref{R:generating} to obtain the
following proposition.
\begin{prop}
Let $D$ be a directed tree with $2n$ vertices obtained from a
basic tree $T$ by replacing every edge of $T$ by a pair of
directed edges going in opposite directions. Then we have that
$\Delta(D)\cong \mathbb{S}^{n-1}$.
\end{prop}
\begin{proof}
Assume that $v_1,v_2,\ldots,v_n$ are leaves of $T$. We label the
rest of the vertices of $T$ with $u_1,u_2,\ldots,u_n$ so that
$v_iu_i \in E(T)$ for all $i=1,2,\ldots,n$. By applying Remark
\ref{R:generating} successively we obtain that the set of
peripheral edges $\{\overrightarrow{v_iu_i}:i=1,2,\ldots,n\}$ is
the unique generating facet of $\Delta(D)$.

\end{proof}
We denote the unique generating facet for a basic tree $T$ by
$G_T$, that is,
$G_T=\{\overrightarrow{v_1u_1},\overrightarrow{v_2u_2},
\ldots,\overrightarrow{v_nu_n}\}$.

Let $D$ be a double directed tree obtained from a tree $T$. We
describe a bijection between generating simplices of $\Delta(D)$
and decompositions of $T$ into basic trees.

 Let $v_1,v_2,\ldots,v_n$ be a fixed linear order of
$V(D)$ and
 choose the first leaf $v\in
V(D)$ in this order. Assume that $N(v)=\{x\}$ and
$N(x)=\{v,y_1,\ldots,y_k\}$. From (c) of Remark \ref{R:generating}
we know that all generating facets of $\Delta(D)$ have to contain
the edge $\overrightarrow{vx}$. Next, we are looking for
generating facets of complex $\Delta(D_0)$ where
$D_0=(D\setminus\{v\})\setminus\{\overrightarrow{y_1v},
\ldots,\overrightarrow{y_kv}\}$. From (b) of Proposition
\ref{P:recursiveshelling} we have that a generating facet of
$\Delta(D_0)$ must contain edges $\overrightarrow{z_1y_1},
\overrightarrow{z_2y_2},\ldots, \overrightarrow{z_ky_k}$ where
$z_i\in N(y_i)$ and $z_i\neq x$. If $deg_T(y_i)=2$ for all
$i=1,2,\ldots,k$, we consider a subtree of $T$ spanned by
$\{v,x,y_1,z_1,\ldots,y_k,z_k\}$. In the case when
$N_T(y_i)=\{x,z_i,u_1,\ldots,u_r\}$, a generating facet of
$\Delta(D) $ that contains
$\{\overrightarrow{vx},\overrightarrow{z_1y_1},\ldots,
\overrightarrow{z_ky_k}\}$ also contains edges
$\overrightarrow{w_ju_j}$ for $j=1,2,\ldots,r$.

By repeating this procedure we obtain a subtree $B_1$ of $T$ such
that
\begin{enumerate}
    \item[(1)] $B_1$ is a basic tree and $v\in V(B_1)$,
    \item[(2)] for any $x\in V(B_1)$ that is not a leaf in
    $B_1$ we have that $d_{B_1}(x)=d_T(x)$,
    \item[(3)] $|V(B_1)|>2$ whenever $|V(T)|>2$.
\end{enumerate}
Note that there can be more possibilities for a basic tree $B_1$,
see Figure 1. If we can not find a subtree $B_1$ that satisfies
the above conditions, then we obtain that $\Delta(D)$ is
contractible. After we choose a basic tree $B_1$ that satisfies
$(1)$--$(3)$ we proceed in the same way with $T'=T\setminus \{x\in
V(B_1):d_{B_1}(x)=d_T(x)\}$. Note that $T'$ is a forest or a tree.

Let $v'$ be the first leaf of $T'$ and let $T_1$ be the maximal
tree of $T'$ that contains $v'$. Now, we are looking for $B_2$, a
subtree of $T_1$ that satisfies $(1)$--$(3)$. If we can decompose
$T$ into $B_1,B_2,\ldots,B_m$ we say that $(B_1,B_2,\ldots,B_{m})$
is an ordered decomposition of $T$ into $m$ basic trees.

An ordered decomposition $(B_1,B_2,\ldots,B_{m})$ of $T$ that
satisfies $(1)$--$(3)$ produces a generating facet $G_{B_1}\cup
G_{B_2}\cup\cdots\cup G_{B_m}$ of $\Delta (D)$.

  \begin{figure}[h!]\label{Figure}
\begin{center}
\includegraphics[scale=0.23]{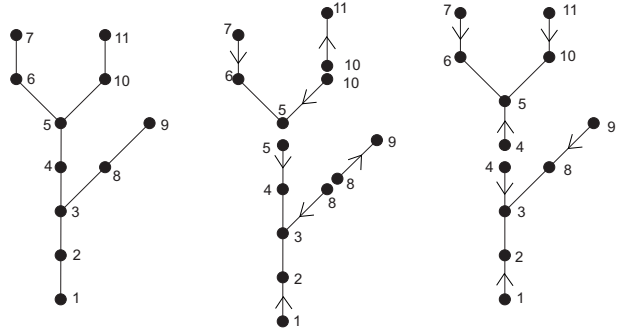}
\caption{\emph{A tree and its decompositions into basic trees.
Oriented edges represent generating facets of $\Delta(D)$}. Note
that $\Delta(D)\simeq \mathbb{S}^5\vee \mathbb{S}^6$}
\label{Sl:prufer}
\end{center}
\end{figure}
\begin{thm} Let $D$ be a double directed tree with $n$ vertices.
Let $\mu_m$ denote the number of ordered decompositions of $D$
into $m$ basic trees. Then we have that
$$\Delta(D)\simeq\bigvee_m\left(\bigvee^{\mu_m}
 \mathbb{S}^{\frac{n+m-3}{2}}\right).$$
\end{thm}
\begin{proof}We described above a bijection between generating sets of
$\Delta(D)$ and ordered decompositions of $D$ that satisfy
$(1)$--$(3)$. Consider such an ordered partition
$(B_1,B_2,\ldots,B_{m})$ with $m$ basic trees. If a basic tree
$B_i$ contains $2s_i$ vertices (and $2s_i-1$ edges) it contains
$s_i$ edges of a generating set of $\Delta(D)$. Then we have $2
s_1-1+2s_2-1+\cdots+2s_{m}-1=n-1\textrm{, and}$ this decomposition
corresponds with
$$s_1+s_2+\cdots+s_{m}-1=\frac{n+m-1}{2}-1$$
dimensional generating facet of $\Delta(D)$.

\end{proof}
\begin{cor}
Among all double directed trees $D$ with $n$ vertices the biggest
dimension of nontrivial homology is
$\left\lceil\frac{n-2}{2}\right\rceil$. Smallest nontrivial
homology for all trees with $n$ vertices appears in the dimension
$n-\lfloor\frac{n}{3}\rfloor-2$.
\end{cor}

\bibliographystyle{plain}
\bibliography{mybib}{}

\end{document}